\documentclass[twocolumn]{article} 
\usepackage{graphicx} 
\usepackage{epsfig,amstext,amsbsy,amsfonts,amsmath}
\usepackage{amsthm,tikz}
\long\def\comment#1{}

\newcommand{\fhi}{\varphi}

\newcommand{\norm}[1]{\lVert#1\rVert}
\newcommand{\ipr}[2]{\left\langle #1, #2 \right\rangle}

\newcommand{\numbersystem}[1]{\mathbb{#1}}

\newcommand{\R}{\numbersystem{R}}

\newcommand{\abs}[1]{\lvert#1\rvert}

\newtheorem{theorem}{Theorem}

  \makeatletter
\def\@normalsize{\@setsize\normalsize{10pt}\xpt\@xpt
\abovedisplayskip 10pt plus2pt minus5pt\belowdisplayskip \abovedisplayskip \abovedisplayshortskip \z@plus3pt\belowdisplayshortskip 6pt plus3pt
minus3pt\let\@listi\@listI}

\def\subsize{\@setsize\subsize{12pt}\xipt\@xipt}
\def\section{\@startsection {section}{1}{\z@}{1.0ex plus
1ex minus .2ex}{.2ex plus .2ex}{\large\bf}}
\def\subsection{\@startsection
   {subsection}{2}{\z@}{.2ex plus 1ex} {.2ex plus .2ex}{\subsize\bf}}
\makeatother

\begin{document}
\date{}

\title{\huge \bf {Designing Optimal Flow Networks}\thanks{This research was supported by an ARC
Linkage Grant with Newmont Australia Limited.}}

\author{M.~G.~Volz\thanks{Department of Electrical and Electronic Engineering,
The University of Melbourne, Victoria 3010,  Australia.} \and M.~Brazil\footnotemark[2] \and  K.~J.~Swanepoel\thanks{Fakult\"at f\"ur
Mathematik,
    Technische Universit\"at Chemnitz,
    D-09107 Chemnitz, Germany} \and D.~A.~Thomas\thanks{Department of Mechanical Engineering,
The University of Melbourne, Victoria 3010,  Australia.}}


\maketitle



{\hspace{1pc} {\it{\small Abstract}}{\bf{\small---We investigate the problem of designing a minimum cost flow network interconnecting $n$
sources and a single sink, each with known locations and flows. The network may contain other unprescribed nodes, known as Steiner points. For
concave increasing cost functions, a minimum cost network of this sort has a tree topology, and hence can be called a Minimum Gilbert
Arborescence (MGA). We characterise the local topological structure of Steiner points in MGAs for linear cost functions. This problem has
applications to the design of drains, gas pipelines and underground mine access.

}}
 }

\section{Introduction}
One of the most important advances in physical network design optimisation since the 1960's has been the development of theory for solving the
\emph{Steiner Minimum Tree (SMT) problem}.
 This problem asks for a shortest
network spanning a given set of points, called \emph{terminals}, in a given metric space. It differs from the minimum spanning tree problem in
that additional points, called \emph{Steiner points}, can be included
 to create a spanning network that is shorter than would
otherwise be possible. This has numerous applications, including the design of telecommunications or transport networks for the problem in the
Euclidean plane (the $l_2$ metric), and the design of microchips for the problem in the rectilinear plane (the $l_1$ metric)~\cite{hwang-1992}.

Gilbert~\cite{gilbert-1967} proposed a generalisation of the SMT problem whereby symmetric non-negative \emph{flows} are assigned between each
pair of terminals. The aim is to find a least cost network interconnecting the terminals, where each edge has an associated total such that the
flow conditions between terminals are satisfied, and Steiner points satisfy Kirchhoff's rule (ie, the net incoming and outgoing flows at each
Steiner point are equal). The cost of an edge is its length multiplied by a non-negative \emph{weight}. The weight is determined by a given
function of the total flow being routed through that edge, where the function satisfies a number of conditions. The \emph{Gilbert network
problem} (GNP) asks for a minimum-cost network spanning a given set of terminals with given flow demands and a given weight function.

An important variation on this problem that we will show to be a special case of the GNP occurs when the  terminals consist of $n$ sources and a
unique sink, and all flows not between a source and the sink are zero. This problem has applications to drainage networks~\cite{lee-1976}, gas
pipelines~\cite{bhaskaran-1979}, and underground mining networks~\cite{brazil-2000}.

If the weight function is concave and increasing, the resulting minimum network has a tree topology, and provides a directed path from each
source to the sink. Such a network can be called an \emph{arborescence}, and we refer to this special case of the GNP as the \emph{Gilbert
arborescence problem} (GAP). Traditionally, the term `arborescence' has been used to describe a rooted tree providing directed paths from the
unique root (source) to a given set of sinks. Here we are interested in the case where the flow directions are reversed, i.e. flow is from $n$
sources to a unique sink. It is clear, however, that the resulting weights for the two problems are equivalent, hence we will continue to use
the term `arborescence' for the latter case. Moreover, if we take the sum of these two cases, and rescale the flows (dividing flows in each
direction by $2$), then again the weights for the total flow on each edge are the same as in the previous two cases. This justifies our claim
that the GAP can be treated as a special case of the GNP. It will be convenient, however, for the remainder of this paper to think of
arborescences as networks with a unique sink.

A \emph{minimum Gilbert arborescence} (MGA) is a (global) minimum-cost arborescence for a given set of terminals and flows, and a given cost
function. All flows in the network are directed towards the unique sink. MGAs have been used to model drainage networks, such as plumbing
networks in buildings~\cite{lee-1976}, as well as gas pipeline networks~\cite{bhaskaran-1979}. In both cases heuristics were provided for
obtaining low-cost solutions. A special complication for these networks is that the cost of an edge depends on the diameter of the pipe, and the
diameter of the pipe depends on its length. This complication will not be considered in this work.

More recently, MGAs have been used to model underground mining networks~\cite{brazil-2000}. Given a set of underground locations, called
\emph{draw points}, and their estimated ore tonnages, the development and haulage costs associated with an underground mine can be minimised by
finding an MGA interconnecting the underground points with a fixed breakout point at the surface. Although the Euclidean metric is  useful in
some situations, more generally additional constraints must be imposed for truck navigability such as a gradient constraint, which can be
handled by altering the metric~\cite{brazil-2000}.

In Section~2 we specify the nature of the weight function that we consider in this paper, and formally define minimum Gilbert networks and
Gilbert arborescences in Minkowski spaces (which generalise Euclidean spaces). In Section~3 we give a general topological characterisation of
Steiner points in such networks, for smooth Minkowski spaces. We then apply this characterisation, in Section~4, to the smooth Minkowski plane
with a linear weight function to show that in this case all Steiner points have degree $3$.

\section{Preliminaries}

The cost functions for the networks we consider in this paper make use of more general norms than simply the Euclidean norm. Hence, we introduce
a generalisation of Euclidean spaces, namely finite-dimensional normed spaces or Minkowski spaces. See \cite{Thompson} for an introduction to
Minkowski geometry.

A \emph{Minkowski space} (or \emph{finite-dimensional Banach space}) is $\R^n$ endowed with a \emph{norm} $\norm{\cdot}$, which is a function
$\norm{\cdot}:\R^n\to\R$ that satisfies
\begin{itemize}
\item $\norm{x}\geq 0$ for all $x\in \R^n$, $\norm{x}=0$ only if $x=0$,
\item $\norm{\alpha x}=\abs{\alpha}\norm{x}$ for all $\alpha\in\R$ and $x\in \R^n$, and
\item $\norm{x+y}\leq\norm{x}+\norm{y}$.
\end{itemize}

We now discuss some aspects of the SMT problem, since this is a special case of the GNP, where all flow are zero. Our terminology for the SMT
problem is based on that used in~\cite{hwang-1992}. Let $T$ be a network interconnecting a set $N = \{p_1,\ldots,p_n\}$ of points, called
\emph{terminals}, in a Minkowski space. A vertex in $T$ which is not a terminal is called a \emph{Steiner point}. Let $G(T)$ denote the
\emph{topology} of $T$, i.e. $G(T)$ represents the graph structure of $T$ but not the embedding of the Steiner points. Then $G(T)$ for a
shortest network $T$ is necessarily a tree, since if a cycle exists, the length of $T$ can be reduced by deleting an edge in the cycle. A
network with a tree topology is called a \emph{tree}, its links are called \emph{edges}, and its nodes are called \emph{vertices}. An edge
connecting two vertices $a,b$ in $T$ is denoted by $ab$, and its  length by $\norm{a-b}$.

The
\emph{splitting} of a vertex is the operation of disconnecting two edges $av,bv$ from a vertex $v$ and connecting $a,b,v$ to a newly created
Steiner point. 
Furthermore, though the positions of terminals are fixed, Steiner points can be subjected to arbitrarily small movements provided the resulting
network is still connected. Such movements are called \emph{perturbations}, and are useful for examining whether the length of a network is
minimal.

A \emph{Steiner tree} (ST) is a tree whose length cannot be shortened by a small perturbation of its Steiner points, even when splitting is
allowed. By convexity, an ST is a minimum-length tree for its given topology. A \emph{Steiner minimum tree} (SMT) is a shortest tree among all
STs. For many Minkowski spaces bounds are known for the maximum possible degree of a Steiner point in an ST, giving useful restrictions on the
possible topology of an SMT. For example, in Euclidean space of any dimension every Steiner point in an ST has degree three. Given a set $N$ of
terminals, the \emph{Steiner problem} (or \emph{Steiner Minimum Tree problem}) asks for an SMT spanning $N$.

Gilbert~\cite{gilbert-1967} proposed the following generalisation of the Steiner problem in Euclidean space, which we now extend to Minkowski
space. Let $T$ be a network interconnecting a set $N = \{p_1,\ldots,p_n\}$ of $n$ terminals in a Minkowski space. For each pair $p_i,p_j,\;i\neq
j$ of terminals, a non-negative \emph{flow} $t_{ij} = t_{ji}$ is given. The cost of an edge $e$ in $T$ is $w(t_e)l_e$, where $l_e$ is the length
of $e$, $t_e$ is the total flow being routed through $e$, and  weight $w(\cdot)$ is a unit cost function satisfying

\begin{eqnarray}
  w(t) &\geq& 0\;\;\;\;\mathrm{and}\;\;\;\;w(t) > 0\;\mathrm{if}\;t>0 \label{eq:cost-ftn-non-neg} \\
  w(t_1 + t_2) &\geq& w(t_1)\;\;\; \forall \;t_2 > 0 \label{eq:cost-ftn-non-dec} \\
  w(t_1 + t_2) &\leq& w(t_1) + w(t_2)\;\;\; \forall \;t_1,t_2 > 0 \label{eq:cost-ftn-trnglr}\\
  w(\cdot)&\mbox{{is}}& \mbox{{concave}} \label{eq:cost-ftn-concave}
\end{eqnarray}

A network satisfying Conditions~(\ref{eq:cost-ftn-non-neg}) - (\ref{eq:cost-ftn-trnglr}) is called a \emph{Gilbert network}. For a given edge
$e$ in $T$, $w(t_e)$ is called the \emph{weight} of $e$, and is also denoted simply by $w_e$. The \emph{total cost} of a Gilbert network $T$ is
the sum of all edge costs, i.e.
\begin{eqnarray*}
  C(T) &=& \sum_{e\in E}w(t_e)l_e
\end{eqnarray*}
\noindent where $E$ is the set of all edges in $T$. A Gilbert network $T$ is a \emph{minimum Gilbert network} (MGN), if $T$ has the minimum cost
of all Gilbert networks spanning the same point set $N$, with the same flow demands $t_{ij}$ and the same cost function $w(\cdot)$. By the
arguments of~\cite{cox-1998}, an MGN always exists in a Minkowski space when Conditions~(\ref{eq:cost-ftn-non-neg}) -
(\ref{eq:cost-ftn-concave}) are assumed for the weight function.

Conditions~(\ref{eq:cost-ftn-non-neg}) - (\ref{eq:cost-ftn-trnglr}) above ensure that the weight function is non-negative, non-decreasing and
triangular, respectively. These are natural conditions for most applications. Unfortunately, the first three conditions alone do not guarantee
that a minimum Gilbert network is a tree. An example of such a minimum network where the flow necessarily splits is given in~\cite{volz-2009}.
(Note that in~\cite{cox-1998}, Condition~(\ref{eq:cost-ftn-trnglr}), which we call the \emph{triangular condition}, was incorrectly interpreted
as concavity of the cost function.)

 The \emph{Gilbert network problem} (GNP) is to find an MGN for a
given terminal set $N$, flows $t_{ij}$ and cost function $w(\cdot)$. Since its introduction in~\cite{gilbert-1967}, various aspects of the GNP
have been studied, although the emphasis has been on discovering geometric properties of MGNs
(see~\cite{cox-1998},~\cite{thomas-2006},~\cite{trietsch-1985},~\cite{trietsch-1999}). As in the Steiner problem, additional vertices can be
added to create a Gilbert network whose cost is less than would otherwise be possible, and these additional points are again called
\emph{Steiner points}. A Steiner point $s$ in $T$ is called \emph{locally minimal} if a perturbation of $s$ does not reduce the cost of $T$. A
Gilbert network is called \emph{locally minimal} if no perturbation of the Steiner points reduces the cost of $T$.

The special case of the Gilbert model that is of interest in this work is when $N = \{p_1,\ldots,p_n\, q\}$ is a set of terminals in a Minkowski
space, where $p_1,\ldots,p_{n}$ are \emph{sources} with respective positive flows $t_1,\ldots,t_{n}$, and $q$ is the \emph{sink}. All flows are
between the sources and the sink; there are no flows between sources. It has been shown in~\cite{thomas-2006} that concavity of the weight
function implies that an MGN of this sort is a tree. Hence we refer to an MGN with this flow structure as a \emph{minimum Gilbert arborescence}
(MGA), and, as mentioned in the introduction, we refer to the problem of constructing such an MGA as the \emph{Gilbert arborescence problem}
(GAP).

If $v_1$ and $v_2$ are two adjacent vertices in a Gilbert arborescence, and flow is from $v_1$ to $v_2$ then we denote the edge connecting the
two vertices by $v_1v_2$.

\section{Characterisation of Steiner Points}\label{section:characterisation}
In this section, we generalise a theorem of Lawlor and Morgan \cite{lawlor-1994} to give a local characterisation of Steiner points in an MGA.
The characterisation in \cite{lawlor-1994} holds for SMTs, which correspond to the case of MGAs with a constant weight function. Their theorem
is formulated for arbitrary Minkowski spaces with differentiable norm. Our proof is based on the proof of Lawlor and Morgan's theorem given in
\cite{swanepoel-1999}. A generalisation to non-smooth norms is contained in~\cite{swanepoel-2007} for SMTs and in \cite{thesis} for MGAs. Such a
generalisation is much more complicated and involves the use of the subdifferential calculus.

We first introduce some necessary definitions relating to Minkowski geometry, in particular with relation to dual spaces. For more details, see
\cite{Thompson}.

We denote the inner product of two vectors $x,y\in\R^n$ by $\ipr{x}{y}$. For any given norm $\norm{\cdot}$, the dual norm $\norm{\cdot}^\ast$ is
defined as follows:
\[\norm{z}^\ast = \sup_{\norm{x}\leq 1}\ipr{z}{x}.\]

We say that a Minkowski space $(\R^n,\norm{\cdot})$ is {\em smooth} if the norm is differentiable at any $x\neq o$, i.e., if
\[\lim_{t\to 0} \frac{\norm{x+th}-\norm{x}}{t} =: f_x(h)\]
exists for all $x,h\in \R^n$ with $x\neq o$. It follows easily that $f_x$ is a linear operator $f_x:\R^n\to\R$ and so can be represented by a
vector $x^\ast\in\R^n$, called the dual vector of $x$, such that $\ipr{x^\ast}{y}=f_x(y)$ for all $y\in\R^n$, and $\norm{x^\ast}^\ast=1$. In
fact $x^\ast$ is just the gradient of the norm at $x$, i.e., $x^\ast=\nabla\norm{x}$.

More generally, even if the norm is not differentiable at $x$, a vector $x^\ast\in \R^n$ is a {\em dual vector} of $x$ if $x^\ast$ satisfies
$\ipr{x^\ast}{x}=\norm{x}$ and $\norm{x^\ast}^\ast=1$. By the Hahn-Banach separation theorem, each non-zero vector in a Minkowski space has at
least one dual vector. A Minkowski space is then smooth if and only if each non-zero vector has a unique dual vector.

A norm is {\em strictly convex} if $\norm{x}=\norm{y}=1$ and $x\neq y$ imply that $\norm{\frac{1}{2}(x+y)}<1$, or equivalently, that the
\emph{unit sphere} \[ S(\norm{\cdot}) = \{x\in\R^n: \norm{x}=1\}\]  does not contain any straight line segment. A norm $\norm{\cdot}$ is smooth
[strictly convex] if and only if the dual norm $\norm{\cdot}^\ast$ is strictly convex [smooth, respectively].

\begin{theorem}\label{theorem:gilb-arb-charac}
Suppose a smooth Minkowski space $(\R^n,\norm{\cdot})$ is given together with  a concave weight function $w$, sources $p_1,\dots,p_n\in\R^n$,
and a single sink $q\in\R^n$, all different from the origin $o$. Let the flow associated with $p_i$ be $t_i$. (See Figure~\ref{Fig:th1}.)
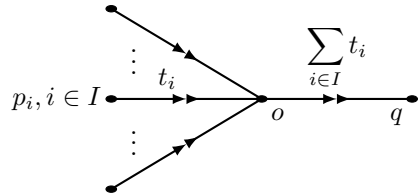
\begin{figure}[htb]
{\qquad\qquad\begin{tikzpicture}[yscale=0.6]
\filldraw (-1,0) coordinate (x) circle (2pt) node[below right] {$o$}; \filldraw (1,0) coordinate (o) circle (2pt) node[below left] {$q$};
\filldraw (-3,2) coordinate (p1) circle (2pt) node[below] {}; \filldraw (-3,0) coordinate (p2) circle (2pt) ; \filldraw (-3,-2) coordinate (pn)
circle (2pt) node[below] {}; \draw (-2.7,-0.8) node {$\vdots$} ; \draw (-2.7,1) node {$\vdots$} ;
\begin{scope}[thick]
\draw[-latex] (p1) -- (-2,1); \draw[latex reversed-] (-2,1) -- (x); \draw[-latex] (p2) node[left] {$p_i, i\in I$} -- (-2,0) node [above left]
{$t_i$}; \draw[latex reversed-] (-2,0) -- (x); \draw[-latex] (pn) -- (-2,-1); \draw[latex reversed-] (-2,-1) -- (x); \draw[-latex] (x) -- (0,0)
node [above] {$\displaystyle\sum_{i\in I} t_i$}; \draw[latex reversed-] (0,0) -- (o);
\end{scope}
\end{tikzpicture}} \caption{A Gilbert network with star topology.} \label{Fig:th1}
\end{figure}
For each $p_i$ let $p_i^\ast$ denote its dual vector, and let $q^\ast$ denote the dual vector of $q$. Then the Gilbert arborescence with edges
$op_i$, $i=1,\dots,n$ and $oq$, where all flows are routed via the Steiner point $o$, is a minimal Gilbert arborescence if and only if
\begin{equation}\label{balancing}
\sum_{i=1}^n w(t_i)p_i^\ast + w(\sum_{i=1}^n t_i)q^\ast = o
\end{equation}
and
\begin{equation}\label{collapsing}
\norm{\sum_{i\in I} w(t_i)p_i^\ast}^\ast\leq w(\sum_{i\in I} t_i)\text{ for all $I\subseteq\{1,\dots,n\}$.}
\end{equation}
\end{theorem}

Note: We think of Condition~\ref{balancing} as a flow-balancing condition at the Steiner point, and Condition~\ref{collapsing} as a condition
that ensures that the Steiner point does not split.

\begin{proof}
$(\Rightarrow)$ We are given that the star is not more expensive than any other Gilbert network with the same sources, sink, flows and weight
function.

In particular, $o$ is the so-called weighted Fermat-Torricelli point of $p_1,\dots,p_n,q$ with weights $t_1,\dots,t_n,\sum_{i=1}^n t_i$,
respectively, which implies the balancing condition \eqref{balancing}. We include a self-contained proof for completeness. If the Steiner point
$o$ is moved to $-te$, where $t\in\R$ and $e\in\R^n$ is a unit vector (in the norm), the resulting arborescence is not better, by the assumption
of minimality. Therefore, the function
\begin{align*}
 \fhi_e(t) &=\sum_{i=1}^n w(t_i)(\norm{p_i+te}-\norm{p_i})\\ &\quad+w(\sum_{i=1}^n t_i)(\norm{q+te}-\norm{q})\geq 0
 \end{align*}
attains its minimum at $t=0$. For $t$ in a sufficiently small neighbourhood of $0$, $p_i+te\neq o$ and $q+te\neq o$, hence $\fhi_e$ is
differentiable. Therefore,
\begin{align*}
0&=\fhi_e'(0) =\lim_{t\to 0} \left(\sum_{i=1}^n w(t_i)\frac{\norm{p_i+te}-\norm{p_i}}{t}\right.\\
& \qquad\qquad \qquad +\left.w(\sum_{i=1}^n t_i)\frac{\norm{q+te}-\norm{q}}{t}\right)\\
&= \sum_{i=1}^n w(t_i)\ipr{p_i^\ast}{e}+w(\sum_{i=1}^n t_i)\ipr{q^\ast}{e}\\
&= \ipr{\sum_{i=1}^n w(t_i)p_i^\ast+w(\sum_{i=1}^n t_i)q^\ast}{e}.
\end{align*}
Since this holds for all unit vectors $e$, \eqref{balancing} follows.

To show \eqref{collapsing} for each $I\subseteq\{1,\dots,n\}$, we may assume without loss of generality that $I\neq\emptyset$ and
$I\neq\{1,\dots,n\}$. Consider the Gilbert network obtained by splitting the Steiner point into two points $o$ and $+te$ ($t\in\R$, $e$ a unit
vector) as follows. Each $p_i$, $i \notin I$, is still adjacent to $o$ with flow $t_i$, and $q$ is joined to $o$ with flow $\sum_{i=1}^n t_i$,
but now each $p_i$, $i\in I$, is adjacent to $te$ with flow $t_i$, and $te$ is adjacent to $o$ with flow $\sum_{i\in I} t_i$, as shown in
Figure~\ref{Fig:split}.

\begin{figure}[tb]
\begin{tikzpicture}[yscale=0.6]
\filldraw (-1,0) coordinate (x) circle (2pt) node[below right] {$te$}; \filldraw (1,0) coordinate (o) circle (2pt) node[below left] {$o$};
\filldraw (-3,2) coordinate (p1) circle (2pt) node[below] {}; \filldraw (-3,0) coordinate (p2) circle (2pt) ; \draw (-3.4,-0.7) node {$p_i, i\in
I$}; \filldraw (-3,-2) coordinate (pn) circle (2pt) node[below] {}; \filldraw (3,2) coordinate (p3) circle (2pt); \draw (3.5,1) node {$p_i,
i\notin I$};
\filldraw (3,0) coordinate (p5) circle (2pt) node[below] {}; \filldraw (3,-2) coordinate (q) circle (2pt) node[below] {$q$}; \draw (-2.7,-0.8)
node {$\vdots$} ; \draw (-2.7,1) node {$\vdots$} ; \draw (2.7,1) node {$\vdots$} ;

\begin{scope}[thick]
\draw[-latex] (p1) -- (-2,1); \draw[latex reversed-] (-2,1) -- (x); \draw[-latex] (p2) -- (-2,0) node [above left] {$t_i$}; \draw[latex
reversed-] (-2,0) -- (x); \draw[-latex] (pn) -- (-2,-1); \draw[latex reversed-] (-2,-1) -- (x); \draw[-latex] (x) -- (0,0) node [above]
{$\displaystyle\sum_{i\in I} t_i$}; \draw[latex reversed-] (0,0) -- (o); \draw[-latex] (p3) -- (2,1) node [below right] {$t_i$}; \draw[latex
reversed-] (2,1) -- (o);
\draw[-latex] (p5) -- (2,0); \draw[latex reversed-] (2,0) -- (o); \draw[-latex] (o) -- (2,-1) node [below left=-6pt] {$\displaystyle\sum_{i=1}^n
t_i$}; \draw[latex reversed-] (2,-1) -- (q);
\end{scope}
\end{tikzpicture}\caption{The Gilbert network obtained by splitting the Steiner point $o$.} \label{Fig:split}
\end{figure}
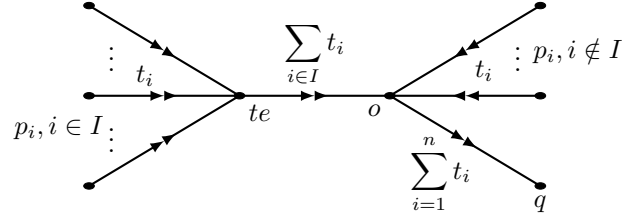

Since the new network cannot be better than the original star, we obtain that for any unit vector $e$, the function
\begin{multline*}
 \psi_e(t)=\sum_{i\in I} w(t_i)(\norm{p_i-te}-\norm{p_i})+w(\sum_{i=1}^n t_i)\abs{t}\geq 0
 \end{multline*}
attained its minimum at $t=0$. Although $\psi_e$ is not differentiable at $0$, we can still calculate as follows:
\begin{align*}
0 &\leq \lim_{t\to 0+}\frac{\psi_e(t)}{t}\\
&= \lim_{t\to0+}\sum_{i\in I} w(t_i)\frac{\norm{p_i-te}-\norm{p_i}}{t}+w(\sum_{i=1}^n t_i)\\
&= \ipr{\sum_{i\in I} w(t_i)p_i^\ast}{-e}+w(\sum_{i=1}^n t_i).
\end{align*}
Therefore, $\ipr{\sum_{i\in I} w(t_i)p_i^\ast}{e}\leq w(\sum_{i=1}^n t_i)$ for all unit vectors $e$, and \eqref{collapsing} follows from the
definition of the dual norm.

\noindent $(\Leftarrow)$ Now assume that $p_1^\ast,p_n^\ast,q$ are dual unit vectors that satisfy \eqref{balancing} and \eqref{collapsing}.
Consider an arbitrary Gilbert arborescence $T$ for the given data. For each $i$, let $P_i$ be the path in $T$ from $p_i$ to $q$, i.e., $P_i =
x_{1}^{(i)}x_2^{(i)}\dots x_{k_i}^{(i)}$, where $x_1^{(i)}=p_i, x_{k_i}^{(i)}=q$, and $x_j^{(i)} x_{j+1}^{(i)}$ are distinct edges of $T$ for
$j=1,\dots,k_i-1$. For each edge $e$ of $T$, let $S_e=\{i:\text{$e$ is on path $P_i$}\}$. Then the flow on $e$ is $\sum_{i\in S_e} t_i$ and the
total cost of $T$ is
\[\sum_{\substack{e=xy\text{ is}\\ \text{an edge of $T$}}} w(\sum_{i\in S_e} t_i)\norm{x-y}.\]
The cost of the star is
\begin{align*}
&\quad\sum_{i=1}^n w(t_i)\norm{p_i}+w(\sum_{i=1}^n t_i)\norm{q} \\
&= \sum_{i=1}^n w(t_i)\ipr{p_i^\ast}{p_i}+w(\sum_{i=1}^n t_i)\ipr{q^\ast}{q} \\
&= \sum_{i=1}^n w(t_i)\ipr{p_i^\ast}{p_i-q}\qquad\text{by \eqref{balancing}}\\
&= \sum_{i=1}^n w(t_i) \sum_{j=1}^{k_i-1}\ipr{p_i^\ast}{x_j^{(i)}-x_{j+1}^{(i)}}
\end{align*}
\begin{align*}
&= \sum_{\substack{e=xy\text{ is}\\ \text{an edge of $T$}}} \ipr{\sum_{i\in S_e} w(t_i) p_i^\ast}{x-y}\\
&\leq \sum_{\substack{e=xy\text{ is}\\ \text{an edge of $T$}}}\norm{\sum_{i\in S_e} w(t_i) p_i^\ast}^\ast\norm{x-y}\\
&\leq \sum_{\substack{e=xy\text{ is}\\ \text{an edge of $T$}}} w(\sum_{i\in S_e} t_i)\norm{x-y} \quad\text{by \eqref{collapsing}}.\qedhere
\end{align*}
\end{proof}

Note that the necessity of the conditions \eqref{balancing} and \eqref{collapsing} holds even if the weight function is not concave. It is only
in the proof of the sufficiency that we need all minimal Gilbert networks with a single source to be arborescences.

\section{Degree of Steiner Points in a Minkowski plane with linear weight function}\label{section:degree}
We now apply the characterisation of the previous section in the two\nobreakdash-\hspace{0pt}dimensional case, assuming further that the weight
function is linear: $w(t)=d+ht$, $d>0, h\geq 0$.

\begin{theorem}
In a smooth Minkowski plane and assuming a linear weight function $w(t)=d+ht$, $d>0, h\geq 0$, a Steiner point in an MGA necessarily has degree
$3$.
\end{theorem}
\begin{proof}
By Theorem~\ref{theorem:gilb-arb-charac}, an MGA with a Steiner point of degree $n+1$ exists in $\R^2$ with a smooth norm $\norm{\cdot}$ if and
only if there exist dual unit vectors $p_1^\ast,\dots,p_n^\ast,q^\ast\in\R^2$ such that
\[ \sum_{i=1}^n (d+ht_i) p_i^\ast + (d+h\sum_{i=1}^n t_i) q^\ast = o \]
and
\begin{multline*}
 \norm{\sum_{i\in I} (d+ht_i)p_i^\ast}^\ast \leq d+h\sum_{i\in I} t_i \quad\text{for all } I\subseteq\{1,\dots,n\}.
 \end{multline*}
Label the $p_i^\ast$ so that they are in order around the dual unit circle. Let $v_i^\ast=(d+ht_i)p_i^\ast$ and $w^\ast=(d+h\sum_{i=1}^n
t_i)q^\ast$. Then the conditions become
\[ v_1^\ast+\dots+v_n^\ast+w^\ast=o,\]
and
\begin{multline}\label{collapsing2}
\norm{\sum_{i\in I} v_i^\ast}^\ast\leq d+h\sum_{i\in I} t_i \quad \text{for all } I\subseteq\{1,\dots,n\}.
\end{multline}
Thus we may think of the vectors $v_1^\ast,\dots,v_n^\ast,w^\ast$ as the edges of a convex polygon with vertices $a_j^\ast=\sum_{i=1}^j
v_i^\ast$, $j=0,\dots,n$ in this order (see Figure~\ref{Fig:polygon}).

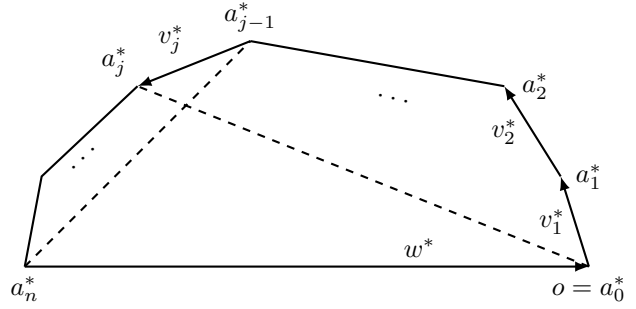
\begin{figure}[htb]
\begin{tikzpicture}[thick, xscale=0.75, yscale=0.6]



\draw[-latex] (0,0) node[below] {$a_n^\ast$} -- (7,0) node[above] {$w^\ast$} -- (10,0) node[below] {$o=a_0^\ast$}; \draw[-latex] (10,0) --
(9.75,1) node[left=0pt] {$v_1^\ast$} -- (9.5,2); \draw[-latex] (9.5,2) node[right=2pt] {$a_1^\ast$} -- (9,3) node[left=2pt] {$v_2^\ast$} --
(8.5,4); \draw (8.5,4) node[right=3pt] {$a_2^\ast$} -- (6.25,4.5) node[below, rotate=25] {$\ddots$} -- (4,5); \draw[-latex] (4,5) node[above]
{$a_{j-1}^\ast$} -- (3,4.5) node[above left=0pt] {$v_j^\ast$}
-- (2,4) node[above left] {$a_j^\ast$} ; 
\draw (2,4) -- (0.3,2) node[below right, rotate=80] {$\ddots$} -- (0,0);
\draw[dashed, thick] (0,0) -- (4,5); \draw[dashed, thick] (10,0) -- (2,4);

\end{tikzpicture}\caption{A polygon with edges corresponding to $v_j^\ast$.}\label{Fig:polygon}
\end{figure}

Assume for the purpose of finding a contradiction that $n > 3$. Then the polygon has at least $4$ sides. Note that the diagonals $a_0^\ast
a_{j}^\ast$ and $a_{j-1}^\ast a_n^\ast$ intersect. Applying the triangle inequality to the two triangles formed by these diagonals and the two
edges $v_j^\ast$ and $w^\ast$ (as illustrated in Figure~\ref{Fig:polygon}), we obtain
\begin{align*}
 \norm{a_j^\ast}^\ast+\norm{a_n^\ast-a_{j-1}^\ast}^\ast
&\geq \norm{v_j^\ast}^\ast+\norm{w^\ast}^\ast\\
&= d+ht_j + d+h\sum_{i=1}^n t_i\\
&= d+h\sum_{i=1}^j t_i + d + h\sum_{i=j}^n t_i\\
&\geq \norm{\sum_{i=1}^j v_i^\ast}^\ast+\norm{\sum_{i=j}^n v_i^\ast}^\ast \qquad\text{by \eqref{collapsing2}}\\
&= \norm{a_j^\ast}^\ast+\norm{a_n^\ast-a_{j-1}^\ast}^\ast.
\end{align*}
Therefore, equality holds throughout, and we obtain equality in the triangle inequality. Since the dual norm is strictly convex, it follows that
$v_j^\ast$ and $w^\ast$ are parallel. This holds for all $j=2,\dots,n-1$. It follows that $p_1^\ast=\dots=p_n^\ast=-q^\ast$. Geometrically this
means that the unit vectors $\frac{1}{\norm{p_i}}p_i$ and $-\frac{1}{\norm{q}}q$ all have the same supporting line on the unit ball. We can
think of this condition on the vectors $p_i$ and $q$ as a generalisation of collinearity to Minkowski space.

Choose a point $s_2$ on the edge $op_i$ such that the line through $s_2$ parallel to $op_n$ intersects the edge $op_1$ in $s_1$, say, with
$s_1\neq o$. See Figure~\ref{Fig:unitball}.
\begin{figure}[tb]
\begin{tikzpicture}[yscale=0.9]
\draw (-2, 1.5) parabola bend (0,2) (2,1.5) -- (2,-1.5) parabola bend (0,-2) (-2,-1.5) -- cycle;
\filldraw (0,0) coordinate (o) circle (2pt) node[below] {$o$}; \filldraw (-4,2) coordinate (p1) circle (2pt) node[below] {$p_1$}; \filldraw
(-4,0) coordinate (p2) circle (2pt) node[below] {$p_2$}; \filldraw (-4,-2) coordinate (pn) circle (2pt) node[below] {$p_n$}; \filldraw (-1,0.5)
coordinate (s1) circle (2pt) node[above] {$s_1$}; \filldraw (-2.5,0) coordinate (s2) circle (2pt) node[above] {$s_2$}; \filldraw (4,1)
coordinate (q) circle (2pt) node[below] {$q$}; \draw (-4,-1) node {$\vdots$} ;

\begin{scope}[thick]
\draw[-latex] (o) -- (3,0.75) node [above] {$\displaystyle\sum_{i=1}^n t_i$}; \draw[latex reversed-] (3,0.75) -- (q); \draw[-latex reversed] (o)
-- (-3,1.5) node [above] {$t_1$}; \draw[latex-] (-3,1.5) -- (p1) ; \draw[-latex] (p2) -- (-3.25,0) node [above] {$t_2$}; \draw[latex reversed-]
(-3.25,0)--(s2); \draw[-latex reversed] (o) -- (-3,-1.5) node [above] {$t_n$}; \draw[latex-] (-3,-1.5) -- (pn) ; \draw[-latex] (s1) --
(-0.5,0.25) node[above right=-2pt] {$t_1+t_2$}; \draw[latex reversed-] (-0.5,0.25) -- (o); \draw[-latex] (s2) -- (-1.75,0.25) node[above]
{$t_2$}; \draw[latex reversed-] (-1.75,0.25) -- (s1); \draw[-latex, dashed] (s2) -- (-1.25,0) node[below left=-1pt] {$t_2$}; \draw[latex
reversed-, dashed] (-1.25,0) -- (o);
\end{scope}
\end{tikzpicture}\caption{Illustration of the proof of Theorem~2 for a unit ball with straight line segments on the boundary.}\label{Fig:unitball}
\end{figure}
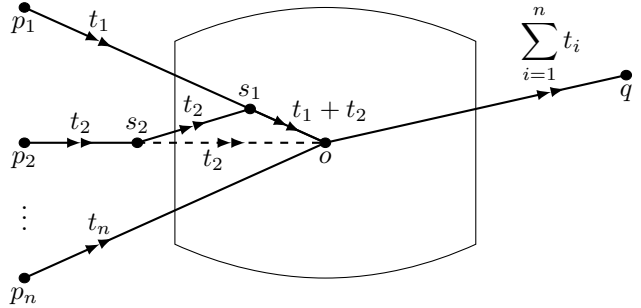
Because of the straight line segments on the boundary of the unit ball,  $\norm{x+y}=\norm{x}+\norm{y}$ for any $x,y$ such that the unit vectors
$\frac{1}{\norm{x}}x$ and $\frac{1}{\norm{y}}y$ lie on this segment. In particular,
\begin{equation}\label{triangleineq}
\norm{s_2-s_1}+\norm{s_1-o}=\norm{s_2-o}.
\end{equation}
Now replace $p_2o$ by the edges $p_2s_2$ and $s_2 s_1$ , replace $p_1o$ by $p_1s_1$ and $s_1o$, and add the flow $t_2$ to $s_1o$. The change in
cost in the new Gilbert arborescence is
\begin{align*}
&\quad (w(t_1)\norm{p_1-s_1}+w(t_1+t_2)\norm{s_1-o}\\
& \qquad\qquad +w(t_2)\norm{p_2-s_2}+w(t_2)\norm{s_2-s_1}) \\
& \qquad - \left(w(t_1)\norm{p_1-o}-w(t_2)\norm{p_2-o}\right)\\
& = -w(t_1)\norm{s_1}-w(t_2)\norm{s_2}+w(t_1+t_2)\norm{s_1}\\
& \qquad +w(t_2)(\norm{s_2}-\norm{s_1}) \qquad\text{by \eqref{triangleineq}}\\
&= (w(t_1+t_2)-w(t_1)-w(t_2))\norm{s_1}\\
&= (d+h(t_1+t_2)-(d+ht_1)-(d+ht_2))\norm{s_1}\\
&= -d\norm{s_1} < 0.
\end{align*}
We have shown that a Steiner point of degree at least $4$ leads to a decrease in cost. In an MGA a Steiner point must then necessarily be of
degree $3$.
\end{proof}

\section{Conclusion}
In this paper we have studied the problem of designing a minimum cost flow network interconnecting $n$ sources and a single sink, each with
known locations and flows, in general finite-dimensional normed spaces. The network may contain other unprescribed nodes, known as Steiner
points. For concave increasing cost functions, a minimum cost network of this sort has a tree topology, and hence can be called a Minimum
Gilbert Arborescence (MGA). We have characterised the local topological structure of Steiner points in MGAs for linear weight functions,
specifically showing that Steiner points necessarily have degree $3$.

\section{Acknowledgments} The authors wish to thank Charl Ras for a number of illuminating discussions on the topic of this paper.

\end{document}